\documentclass[12pt, notitlepage]{amsart}

\usepackage{amssymb}
\usepackage{amscd}
\usepackage{amsfonts}
\usepackage{setspace}
\usepackage{version}
\usepackage{endnotes}
\usepackage{float}

\newtheorem{theorem}{Theorem}[section]
\newtheorem{lemma}[theorem]{Lemma}
\newtheorem{proposition}[theorem]{Proposition}

\newtheorem*{nonumtheorem}{Theorem}
\newtheorem{corollary}[theorem]{Corollary}

\newtheorem*{claim}{Claim}

\theoremstyle{remark}
\newtheorem{remark}[theorem]{Remark}

\theoremstyle{definition}

\theoremstyle{example}

\numberwithin{equation}{section}

\newcommand{\Z}{\mathbb{Z}}
\newcommand{\R}{\mathbb{R}}

\newcommand{\C}{\mathbb{C}}

\newcommand{\e}{{\rm e}}
\newcommand{\CX}{\mathcal{X}}
\newcommand{\ee}{\varepsilon}
\newcommand{\1}{\mathbf{1}}

\newcommand{\CK}{\mathcal{K}}
\newcommand{\CQ}{\mathcal{Q}}

\begin{document}

\title{Gowers norms of multiplicative functions in progressions on average}

\author{Xuancheng Shao}
\address{Mathematical Institute\\ Radcliffe Observatory Quarter\\ Woodstock Road\\ Oxford OX2 6GG \\ United Kingdom}
\email{Xuancheng.Shao@maths.ox.ac.uk}
\thanks{XS was supported by a Glasstone Research Fellowship.}

\maketitle

\begin{abstract}
Let $\mu$ be the M\"{o}bius function and let $k \geq 1$. We prove that the Gowers $U^k$-norm of $\mu$ restricted to progressions $\{n \leq X: n\equiv a_q\pmod{q}\}$ is $o(1)$ on average over $q\leq X^{1/2-\sigma}$ for any $\sigma > 0$, where $a_q\pmod{q}$ is an arbitrary residue class with $(a_q,q) = 1$. This generalizes the Bombieri-Vinogradov inequality for $\mu$, which corresponds to the special case $k=1$.
\end{abstract}

\section{Introduction}

A basic problem in analytic number theory is to understand the distribution of primes, or other related arithmetic functions such as the M\"{o}bius function $\mu$ and the Liouville function $\lambda$, in arithmetic progressions when the modulus is relatively large. In this direction, the Bombieri-Vinogradov inequality leads us almost half way to the ultimate goal, if we average over the moduli.

\begin{nonumtheorem}[Bombieri-Vinogradov]
Let $X, Q \geq 2$, and let $A \geq 2$. Assume that $Q \leq X^{1/2}(\log X)^{-B}$ for some sufficiently large $B = B(A)$. Then for all but at most $Q(\log X)^{-A}$ moduli $q \leq Q$, we have
\[ \sup_{(a,q) = 1} \left| \sum_{\substack{n \leq X \\ n \equiv a\pmod{q}}} \Lambda(n) - \frac{1}{\varphi(q)} \sum_{n \leq X} \Lambda(n) \right| \ll_A \frac{X}{Q(\log X)^A}. \]
The same statement holds for the M\"{o}bius function $\mu$ and the Liouville function $\lambda$.
\end{nonumtheorem}

See~\cite[Chapter 17]{IK04} for its proof and applications. In this paper, we investigate a higher order generalization of the Bombieri-Vinogradov inequality, which measures more refined distributional properties. This higher order version involves Gowers norms, a central tool in additive combinatorics. We refer the readers to~\cite[Chapter 11]{TV06} for the basic definitions and applications. In particular, $\|f\|_{U^k(Y)}$ stands for the $U^k$-norm of the function $f$ on the interval $[0,Y]\cap \Z$.

For any arithmetic function $f: \Z \to \C$ and any residue class $a \pmod q$, denote by $f(q\cdot + a)$ the function $m \mapsto f(qm+a)$. Precisely we study the Gowers $U^k$-norm of $f$ restricted to progressions $\{n \leq X: n \equiv a \pmod{q}\}$, i.e. the $U^k$-norm of the functions $f(q\cdot + a)$ on $[0,X/q] \cap \Z$.

\begin{corollary}\label{cor:BV-nil}
Let $X, Q \geq 2$, let $k$ be a positive integer, let $A \geq 2$, and let $\ee > 0$. Assume that $Q \leq X^{1/2}(\log X)^{-B}$ for some sufficiently large $B = B(k,A,\ee)$. Then for all but at most $Q(\log X)^{-A}$ moduli $q \leq Q$, we have
\[ \sup_{\substack{0 \leq a < q \\ (a,q) = 1}} \|\mu(q \cdot + a)\|_{U^k(X/q)} \leq \ee. \]
The same statement holds for the Liouville function $\lambda$.
\end{corollary}

The Bombieri-Vinogradov inequality is the $k=1$ case of Corollary~\ref{cor:BV-nil} (qualitatively), since the $U^1$-norm of a function is the same as the absolute value of its average. By the inverse theorem for Gowers norms~\cite{GTZ}, Corollary~\ref{cor:BV-nil} is a straightforward consequence of the following result.

\begin{theorem}\label{thm:BV-nil}
Let $X, Q \geq 2$ be parameters with $10Q^2 \leq X$. Associated to each $Q \leq q < 2Q$ we have:
\begin{enumerate}
\item a residue class $a_q \pmod q$ with $0 \leq a_q < q$, $(a_q,q) = 1$;
\item a nilmanifold $G_q/\Gamma_q$ of dimension at most some $d \geq 1$, equipped with a filtration $(G_q)_{\bullet}$ of degree at most some $s \geq 1$ and a $(\log X)$-rational Mal'cev basis $\CX_q$;
\item a polynomial sequence $g_q: \Z \to G_q$ adapted to $(G_q)_{\bullet}$;
\item a Lipschitz function $\varphi_q: G_q/\Gamma_q \to \C$ with $\|\varphi_q\|_{\text{Lip}(\CX_q)} \leq 1$.
\end{enumerate}
Let $\psi_q: \Z \to \C$ be the function defined by $\psi_q(n) = \varphi_q(g_q(n)\Gamma_q)$. Then for any $A \geq 2$, the bound
\begin{equation}\label{eq:q-cancel} 
\left| \sum_{\substack{n \leq X \\ n \equiv a_q\pmod{q}}} \mu(n) \psi_q((n-a_q)/q) \right| \ll_{A,d,s} \frac{X}{Q} \cdot \frac{\log\log X}{\log (X/Q^2)} 
\end{equation}
holds for all but at most $Q(\log X)^{-A}$ moduli $Q \leq q < 2Q$. The same statement holds for the Liouville function $\lambda$.
\end{theorem}

See~\cite{GT-nilsequence} for the precise definitions of nilmanifolds and the associated data appearing in the statement. To avoid confusions later on, we point out that the Lipschitz norm is defined by
\[ \|\varphi_q\|_{\text{Lip}(\CX_q)} = \|\varphi_q\|_{\infty} + \sup_{x \neq y} \frac{|\varphi_q(x) - \varphi_q(y)|}{d(x,y)}, \]
where $d(\cdot,\cdot)$ is the metric induced by $\CX_q$. In particular $\|\varphi_q\|_{\infty} \leq \|\varphi_q\|_{\text{Lip}(\CX_q)}$. 

To understand this paper, however, it is not essential to know these definitions, as long as one is willing to accept certain results about nilsequences as black boxes, many of which can be found in~\cite{GT-nilsequence}. The readers are thus encouraged to consider the following special case when the nilmanifolds are the torus $\R/\Z$, the polynomial sequences are genuine polynomials of degree at most $s$, and the Lipschitz functions are $\varphi(x) = \e(x) = e^{2\pi i x}$.

\begin{nonumtheorem}[Main theorem, special case]
Let $X, Q \geq 2$ be parameters with $10Q^2 \leq X$, and let $s \geq 1$. Then for any $A \geq 2$, the bound
\[
\sup_{\substack{0 \leq a < q \\ (a,q) = 1}} \sup_{\alpha_1,\cdots,\alpha_s \in \R} \left| \sum_{\substack{n \leq X \\ n \equiv a\pmod{q}}} \mu(n) \e(\alpha_sn^s + \cdots + \alpha_1n) \right| \ll_{A,s} \frac{X}{Q} \cdot \frac{\log\log X}{\log (X/Q^2)} 
\]
holds for all but at most $Q(\log X)^{-A}$ moduli $Q \leq q < 2Q$. The same statement holds for the Liouville function $\lambda$.
\end{nonumtheorem}

Without restricting to arithmetic progressions (i.e. when $Q = O(1)$), the discorrelation between the M\"{o}bius function and nilsequences was studied by Green and Tao~\cite{GT-mobius-nil}, as part of their program to count the number of solutions to linear equations in prime variables.

The rest of the paper is organized as follows. In Section~\ref{sec:reduction} we reduce Theorem~\ref{thm:BV-nil} to the minor arc case (Proposition~\ref{prop:BV-equidist}). This reduction process is summarized in Lemma~\ref{lem:reduction2}, using a factorization theorem for nilsequences~\cite[Theorem 1.19]{GT-nilsequence}. In fact, one can obtain analogues of Theorem~\ref{thm:BV-nil} for all $1$-bounded multiplicative functions satisfying the Bombieri-Vinogradov estimate, such as indicator functions of smooth numbers (see~\cite{FT96,Har12} and the references therein). See~\cite{FH} for a previous work on Gowers norms of multiplicative functions, and also~\cite{Mat14} for a generalization to some not necessarily bounded multiplicative functions. However, we will not seek for such generality here since any such result can be easily deduced from Lemma~\ref{lem:reduction2} and Proposition~\ref{prop:BV-equidist} as needed.

The rest of the argument applies to all bounded multiplicative functions. In Section~\ref{sec:BV-equidist} we consider the minor arc case using an orthogonality criterion. The idea, going back to Montgomery-Vaughan~\cite{MV77} and K\'{a}tai~\cite{Kat86} (see also~\cite{BSZ13,Harper}), is that one can make do with type-II estimates (or bilinear estimates) in a very restricted range when dealing with bounded multiplicative functions. This is the reason that we are unable to prove Theorem~\ref{thm:BV-nil} for the primes, which would require type-II estimates in an inaccessible range, and also the reason that one saves no more than $\log X$ in the bound~\eqref{eq:q-cancel}. In fact, to get this saving we use a quantitatively superior argument of Ramar{\'e}~\cite{Ram09}, which received a lot of attention recently~\cite{MRT15,Gre16} following its use in Matom\"{a}ki and Radziwi{\l\l}'s recent breakthrough~\cite{MR16}. Finally the required type-II estimates will be proved in Section~\ref{sec:type-II}.

\section{Technical reductions}\label{sec:reduction}

In this section, we reduce Theorem~\ref{thm:BV-nil} to the following minor arc, or equidistributed, case. See~\cite[Definition 1.2]{GT-nilsequence} for the precise definition about equidistribution of nilsequences.

\begin{proposition}\label{prop:BV-equidist}
Let $X, Q \geq 2$ be parameters with $10Q^2 \leq X$. Let $\eta \in (0,1/2)$. Let $\CQ \subset [Q, 2Q)$ be an arbitrary subset. Associated to each $q \in \CQ$ we have:
\begin{enumerate}
\item a residue class $a_q \pmod q$ with $0 \leq a_q < q$, $(a_q,q) = 1$, and an arbitrary interval $I_q \subset [0,X]$;
\item a nilmanifold $G_q/\Gamma_q$ of dimension at most some $d \geq 1$, equipped with a filtration $(G_q)_{\bullet}$ of degree at most some $s \geq 1$ and an $\eta^{-c}$-rational Mal'cev basis $\CX_q$ for some sufficiently small $c = c(d,s) > 0$;
\item a polynomial sequence $g_q: \Z \to G_q$ adapted to $(G_q)_{\bullet}$ such that $\{g_q(m)\}_{1 \leq m \leq X/q}$ is totally $\eta$-equidistributed;
\item a Lipschitz function $\varphi_q: G_q/\Gamma_q \to \C$ with $\|\varphi_q\|_{\infty} \leq 1$ and $\int \varphi_q = 0$.
\end{enumerate}
Let $\psi_q: \Z \to \C$ be the function defined by $\psi_q(n) = \varphi_q(g_q(n)\Gamma_q)$. Let $f: \Z \to \C$ be a multiplicative function with $|f(n)| \leq 1$. Then
\begin{align*} 
&  \sum_{q \in \CQ} \left| \sum_{\substack{n \in I_q \\ n\equiv a_q\pmod{q}}} f(n) \psi_q( (n-a_q)/q ) \right| \\
\ll & \frac{\log \eta^{-1}}{\log (X/Q^2)} \sum_{q \in \CQ} \left(\frac{|I_q|}{q} + 1\right) + \eta^c X \log X \max_{q \in \CQ} \|\varphi_q\|_{\text{Lip}(\CX_q)}, 
\end{align*}
for some constant $c = c(d,s) > 0$.
\end{proposition}

Thus one obtains a saving of (at most) $\log \eta^{-1}/\log (X/Q^2)$ compared to the trivial bound. The attentive reader may notice an extra factor $\log X$ in the second term of the bound, which prevents one from taking any $\eta = o(1)$ and still getting a nontrivial estimate. This extra factor mainly comes from the type-II estimate (Lemma~\ref{lem:typeII}); see the comments after its statement. It won't be a concern for us since we will take $\eta$ to be a large negative power of $\log X$.

To deduce Theorem~\ref{thm:BV-nil} from Proposition~\ref{prop:BV-equidist}, we may assume that $A$ is sufficiently large depending on $d,s$, that $X$ is sufficiently large depending on $A,d,s$, and that $Q \leq X^{1/2}(\log X)^{-B}$ for some sufficiently large $B = B(A)$, since otherwise the bound~\eqref{eq:q-cancel} is trivial. In particular, it suffices to establish the bound $O_{A}(Q (\log X)^{-2A})$ for the number of exceptional moduli. 

\subsection{Reducing to completely multiplicative functions}\label{sec:BV-nil-complete-mult}

The first technical step of the reduction is to pass from the M\"{o}bius function $\mu$ to its completely multiplicative cousin $\lambda$. In this subsection we deduce Theorem~\ref{thm:BV-nil} for  $\mu$, assuming that it has already been proved for the Liouville function $\lambda$. This step is summarized in the following lemma.

\begin{lemma}\label{lem:reduction1}
Let $X \geq 2$ be large, and let $a \pmod q$ be a residue class with $(a,q) = 1$. Let $\ee \in (0,1)$, and assume that $q \leq \ee X^{1/2}(\log X)^{-3}$.  Let $f: \Z \to \C$ be a multiplicative function with $|f(n)| \leq 1$, and let $f': \Z \to \C$ be the completely multiplicative function defined by $f'(p) = f(p)$ for each prime $p$. Let $c: \Z \to \C$ be an arbitrary function with $|c(n)| \leq 1$. If
\[ \left| \sum_{\substack{n \leq X \\ n \equiv a\pmod{q}}} f(n) c(n) \right| \geq \ee \frac{X}{q}, \]
then there is a positive integer $\ell \ll \ee^{-3}$ with $(\ell, q)=1$, such that
\[ \left| \sum_{\substack{n \leq X/\ell \\ n \equiv a\ell^{-1} \pmod{q}}} f'(n) c(\ell n) \right| \gg \ee \frac{X}{\ell q}. \]
\end{lemma}
 
To deduce Theorem~\ref{thm:BV-nil} for $\mu$, apply Lemma~\ref{lem:reduction1} with $f=\mu$ (so that $f' = \lambda$) and $\ee = C \log\log X/\log (X/Q^2)$ for some large constant $C$ depening on $A$. For each $q$ satisfying
\begin{equation}\label{eq:reduction1-badq} 
\left| \sum_{\substack{n \leq X \\ n \equiv a_q\pmod{q}}} \mu(n) \psi_q((n-a_q)/q) \right| \geq \ee \frac{X}{q},
\end{equation}
Lemma~\ref{lem:reduction1} produces a positive integer $\ell = \ell_q \ll \ee^{-3}$ with $(\ell_q,q) = 1$, such that
\[ \left| \sum_{\substack{n \leq X/\ell_q \\ n \equiv a_q\ell_q^{-1} \pmod{q}}} \lambda(n) \psi_q((\ell_q n - a_q)/q) \right| \gg \ee \frac{X}{\ell_q q}. \]
For each $\ell \ll \ee^{-3}$, apply Theorem~\ref{thm:BV-nil} for $\lambda$, with $X$ replaced by $X/\ell$, to conclude that there are at most $Q(\log X)^{-3A}$ moduli $q$ satisfying~\eqref{eq:reduction1-badq} with $\ell_q = \ell$. It follows that the total number of moduli $q$ satisfying~\eqref{eq:reduction1-badq} is $O(\ee^{-3} Q (\log X)^{-3A}) = O(Q(\log X)^{-2A})$, as desired.
 
In the remainder of this subsection, we give the rather standard proof of Lemma~\ref{lem:reduction1}, starting with a basic lemma. 

\begin{lemma}\label{lem:complete-mult-g}
Let $f: \Z \to \C$ be a multiplicative function with $|f(n)| \leq 1$, and let $f': \Z \to \C$ be the completely multiplicative function defined by $f'(p) = f(p)$ for each prime $p$. Let $g$ be the multiplicative function with $f = f'*g$. Then for any $N \geq 2$ we have
\[ \sum_{n \geq N} \frac{|g(n)|}{n} \ll N^{-1/2}(\log N)^2, \ \  \sum_{n \leq N} |g(n)|  \ll N^{1/2} (\log N)^2. \]
\end{lemma}

\begin{proof}
It is easy to see that $g(p) = 0$ and $|g(p^k)| \leq 2$ for every prime $p$. Set $\sigma = 1/2 + 1/(10\log N)$ so that $\sigma \in (1/2,1)$ and $N^{\sigma} \asymp N^{1/2}$. By Rankin's trick we have
\[ \sum_{n \geq N} \frac{|g(n)|}{n} \leq \sum_n \frac{|g(n)|}{n} \left(\frac{n}{N}\right)^{1-\sigma} \ll N^{-1/2} \sum_n |g(n)|n^{-\sigma}, \]
and similarly
\[ \sum_{n \leq N} |g(n)| \leq \sum_n |g(n)| \left(\frac{N}{n}\right)^{\sigma} \ll N^{1/2} \sum_n |g(n)|n^{-\sigma}. \]
Thus it suffices to establish the bound
\[ \sum_n |g(n)|n^{-\sigma} \ll (\log N)^2. \]
We may write the Dirichlet series associated to $|g|$ in terms of its Euler product:
\[ \sum_n |g(n)|n^{-\sigma} = \prod_p \left(1 + |g(p^2)|p^{-2\sigma} + |g(p^3)|p^{-3\sigma} + \cdots \right). \]
Since $|g(p^k)| \leq 2$, we may bound it by
\[ \prod_p \left(1 + p^{-2\sigma} + p^{-4\sigma} + \cdots \right)^2 \left(1 + p^{-3\sigma} + p^{-6\sigma} + \cdots \right)^2  = \zeta(2\sigma)^2 \zeta(3\sigma)^2. \]
Since $\zeta(3\sigma) \ll 1$ and $\zeta(2\sigma) \ll (2\sigma-1)^{-1}$, the desired bound follows immediately.
\end{proof}

\begin{proof}[Proof of Lemma~\ref{lem:reduction1}]
Write $f = f' * g$ for some multiplicative function $g$. We have
\[ \sum_{\substack{n \leq X \\ n \equiv a\pmod{q}}} f(n) c(n) = \sum_{\substack{\ell n\leq X \\ \ell n \equiv a\pmod{q}}} f'(n) g(\ell) c(\ell n) = \sum_{\substack{\ell \leq X \\ (\ell, q)=1}} g(\ell) \left(\sum_{\substack{n \leq X/\ell \\ n \equiv a\ell^{-1} \pmod{q}}} f'(n) c(\ell n) \right).  \]
Let $L \geq 2$ be a parameter. Using the trivial bound $O(X/\ell q + 1)$ for the inner sum, we may apply Lemma~\ref{lem:complete-mult-g} to bound the total contributions  from those terms with $\ell \geq L$ by
\[ \frac{X}{q} \sum_{\ell \geq L} \frac{|g(\ell)|}{\ell} + \sum_{\ell \leq X} |g(\ell)| \ll \frac{X}{qL^{1/2}} (\log L)^2 + X^{1/2} (\log X)^2. \]
We may choose $L \ll \ee^{-3}$ such that the first term above is negligible compared to the lower bound $\ee X/q$, and the second term is already negligible compared to $\ee X/q$ by the assumption on $q$. It follows that 
\[ \sum_{\substack{\ell \leq L \\ (\ell, q)=1}} |g(\ell)| \left| \sum_{\substack{n \leq X/\ell \\ n \equiv a\ell^{-1} \pmod{q}}} f'(n) c(\ell n) \right| \gg \ee \frac{X}{q}. \]
Since $\sum |g(\ell)|\ell^{-1} \ll 1$, there is some $\ell \leq L$ with $(\ell,q) = 1$ such that
\[ \left| \sum_{\substack{n \leq X/\ell \\ n \equiv a\ell^{-1} \pmod{q}}} f'(n) c(\ell n) \right|\gg \ee \frac{X}{\ell q}. \]
This completes the proof of the lemma.
\end{proof}

\subsection{Reducing to equidistributed nilsequences}\label{sec:BV-nil-general}

We now use the factorisation theorem~\cite[Theorem 1.19]{GT-nilsequence} to reduce arbitrary nilsequences to equidistributed ones. This step is summarized in the following lemma, the proof of which is similar to arguments in~\cite[Section 2]{GT-mobius-nil}.

\begin{lemma}\label{lem:reduction2}
Let $X \geq 2$ be large, and let $a \pmod q$ be a residue class with $0 \leq a<q$, $(a,q) = 1$. Let $\ee \in (0,1/2)$. Let $f: \Z \to \C$ be a completely multiplicative function with $|f(n)| \leq 1$. Given
\begin{itemize}
\item a nilmanifold $G/\Gamma$ of dimension at most some $d \geq 1$, equipped with a filtration $G_{\bullet}$ of degree at most some $s \geq 1$ and a $M_0$-rational Mal'cev basis $\CX$ for some $M_0 \geq \ee^{-1}$;
\item a polynomial sequence $g: \Z \to G$ adapted to $G_{\bullet}$;
\item and a Lipschitz function $\varphi: G/\Gamma \to \C$ with $\|\varphi\|_{\text{Lip}(\CX)} \leq 1$,
\end{itemize}
let $\psi: \Z \to \C$ be the function defined by $\psi(n) = \varphi(g(n)\Gamma)$. Assume that
\[ \left| \sum_{\substack{n \leq X \\ n \equiv a\pmod{q}}} f(n) \psi((n-a)/q) \right| \geq \ee \frac{X}{q}. \]
For any $A \geq 2$ large enough depending on $d,s$, we may find $M_0 \leq M \leq M_0^{O_{A}(1)}$, an interval $I \subset [0,X]$ with $|I| \gg X/M^3$, a positive integer $q'$ with $q \mid q'$ and $q' \leq qM$, a residue class $a'\pmod{q'}$ with $0 \leq a' < q'$, $(a',q') = 1$, and moreover
\begin{itemize}
\item a nilmanifold $G'/\Gamma'$ of dimension at most $d$, equipped with a filtration $G'_{\bullet}$ of degree at most $s$ and a $M^{O_{d,s}(1)}$-rational Mal'cev basis $\CX'$;
\item a polynomial sequence $g': \Z \to G'$ adapted to $G'_{\bullet}$ such that $\{g'(m)\}_{1 \leq m \leq X/q}$ is totally $M^{-A}$-equidistributed;
\item and a Lipschitz function $\varphi': G'/\Gamma' \to \C$ with $\|\varphi'\|_{\infty} \leq 1$ and $\|\varphi'\|_{\text{Lip}(\CX')} \leq M^{O_{d,s}(1)}$,
\end{itemize}
such that
\[ \left| \sum_{\substack{n \in I\\ n \equiv a' \pmod{q'}}} f(n) \psi'((n-a')/q') \right| \gg \ee \frac{|I|}{q'}, \]
where $\psi' : \Z \to \C$ is the function defined by $\psi'(n) = \varphi'(g'(n)\Gamma')$. 
\end{lemma}

To deduce Theorem~\ref{thm:BV-nil} for $\lambda$ from Proposition~\ref{prop:BV-equidist}, apply Lemma~\ref{lem:reduction2} with $f=\lambda$, $\ee = C \log\log X/\log (X/Q^2)$ for some large constant $C$ depending on $A$, and $M_0 = (\log X)^C$ for some large constant $C$ depending on $d,s$. For each $q$ satisfying
\begin{equation}\label{eq:reduction2-badq} 
\left| \sum_{\substack{n \leq X \\ n \equiv a_q\pmod{q}}} \lambda(n) \psi_q((n-a_q)/q) \right| \geq \ee \frac{X}{q},
\end{equation}
Lemma~\ref{lem:reduction2} produces $M_q, I_q, a'\pmod{q'}, G'/\Gamma', \psi' = \varphi' \circ g'$, all of which depending on $q$ (and some of these dependences are suppressed for notational convenience), such that
\[ \left| \sum_{\substack{n \in I_q\\ n \equiv a' \pmod{q'}}} \lambda(n) \psi'((n-a')/q') \right| \gg \ee \frac{|I_q|}{q'}. \]
Divide the possible values of $M_q$ into $O_A(1)$ subintervals of the form $[M^{1/2}, M]$ with $M_0 \leq M \leq M_0^{O_{A}(1)}$. Given $M$, let $\CQ = \CQ_M$ be the set of moduli $q \in \CQ$ with $M^{1/2} \leq M_q \leq M$, and let $\CQ' = \CQ'_M$ be the set of $q'$ arising from $q \in \CQ$. It suffices to show that
\[ |\CQ| \ll Q(\log X)^{-2A}. \]
Since $q'/q$ is a positive integer at most $M$, each $q'$ occurs with multiplicity at most $M$. Thus $|\CQ| \leq M |\CQ'|$. Before applying Proposition~\ref{prop:BV-equidist} we need to ensure that each $\varphi'$ has average $0$. For $q' \in \CQ'$ either
\begin{equation}\label{eq:red2-type1} 
\left| \sum_{\substack{n \in I_q\\ n \equiv a' \pmod{q'}}} \lambda(n)  \right| \gg \ee \frac{|I_q|}{q'} \gg \frac{X}{q'(\log X)^{O_A(1)}} 
\end{equation}
or
\begin{equation}\label{eq:red2-type2}
\left| \sum_{\substack{n \in I_q\\ n \equiv a' \pmod{q'}}} \lambda(n) \left(\psi'((n-a')/q') - \int \varphi'\right) \right| \gg \ee \frac{|I_q|}{q'}.
\end{equation}

To bound the number of $q'$ satisfying~\eqref{eq:red2-type1}, note that the Bombieri-Vinogradov inequality (for $\lambda$) is applicable since $q' \leq 2QM \leq X^{1/2}(\log X)^{-B/2}$ (recall the assumption that $Q \leq X^{1/2}(\log X)^{-B}$ for some large $B$). By choosing $b$ large enough, we may ensure that the number of $q'$ satisfying~\eqref{eq:red2-type1} is at most $QM^{-A}$. 

Now let $\CQ_2' \subset \CQ'$ be the set of $q' \in \CQ'$ satisfying~\eqref{eq:red2-type2}.  To bound the size of $\CQ_2'$, we apply Proposition~\ref{prop:BV-equidist} after replacing each $\varphi'$ by $\varphi' - \int\varphi'$ and dyadically dividing the possible values of $q'$. This leads to
\[ \ee \sum_{q' \in \CQ_2'} \frac{|I_q|}{q'} \ll \frac{\log M^{A}}{\log (X/Q^2M^2)}  \sum_{q' \in \CQ_2'}\left(\frac{|I_q|}{q'} + 1\right) + M^{-cA} X(\log X)^2 M^{O_{d,s}(1)}, \]
for some $c = c(d,s) > 0$. The first term on the right can be made negligible compared to the left hand side, if the constant $C$ in the choice of $\ee$ is taken large enough in terms of $A$. Hence
\[ \ee\frac{X}{M^3} \cdot \frac{|\CQ_2'|}{QM} \ll \ee \sum_{q' \in \CQ'} \frac{|I_q|}{q'} \ll M^{-cA + O_{d,s}(1)} X. \]
It follows that $|\CQ_2'| \ll Q M^{-cA + O_{d,s}(1)}$. Combining the estimates for the two types of $q'$ together, we obtain
\[ |\CQ| \leq M|\CQ'| \ll QM^{-cA + O_{d,s}(1)} \ll Q(\log X)^{-2A}, \]
if the constant $C$ in the choice of $M$ is large enough depending on $d,s$. This completes the deduction of Theorem~\ref{thm:BV-nil}.

\begin{proof}[Proof of Lemma~\ref{lem:reduction2}]
Let $C$ be a large constant (depending on $d,s$). Apply the factorisation theorem~\cite[Theorem 1.19]{GT-nilsequence} to find $CM_0 \leq M \leq M_0^{O_A(1)}$, a rational subgroup $\widetilde{G} \subset G$, a Mal'cev basis $\widetilde{\CX}$ for $\widetilde{G}/\widetilde{\Gamma}$ (where $\widetilde{\Gamma} = \Gamma\cap \widetilde{G}$) in which each element is an $M$-rational combination of the elements of $\CX$, and a decomposition $g = s \widetilde{g} \gamma$ into polynomial sequences $s,\widetilde{g},\gamma: \Z \to G$ with the following properties:
\begin{enumerate}
\item $s$ is $(M, X/q)$-smooth in the sense that $d(s(n),\text{id}) \leq M$ and $d(s(n), s(n-1)) \leq qM/X$ for each $1 \leq n \leq X/q$;
\item $\widetilde{g}$ takes values in $\widetilde{G}$, and moreover $\{\widetilde{g}(n)\}_{1 \leq n \leq X/q}$ is totally $M^{-CA}$-equidistributed in $\widetilde{G}/\widetilde{\Gamma}$ (using the metric induced by the Mal'cev basis $\widetilde{\CX}$);
\item $\gamma$ is $M$-rational in the sense that for each $n \in \Z$, $\gamma(n)^r \in \Gamma$ for some $1 \leq r \leq M$. Moreover, $\gamma$ is periodic with period $t \leq M$.
\end{enumerate}
We may assume that $X \geq qM^3$, since otherwise the conclusion holds trivially. After a change of variables $n = qm+a$, we may rewrite the assumption as
\[ \left| \sum_{m \leq X/q} f(qm+a) \psi(m) \right| \gg \ee \frac{X}{q}. \]
Dividing $[0,X/q]$ into $O(M^2)$ intervals of equal length and then further divide them into residue classes modulo $t$, we may find an interval $J \subset [0,X]$ with $|J| \asymp X/M^2$ and some residue class $b\pmod{t}$, such that
\begin{equation}\label{eq:q-bad-prog} 
\left| \sum_{\substack{m \equiv b\pmod{t} \\ qm+a \in J}} f(qm+a) \varphi(s(m) \widetilde{g}(m) \gamma(m)) \right| \gg \ee \frac{|J|}{qt}.
\end{equation}
Pick any $m_0$ counted in the sum (i.e. $m_0 \equiv b\pmod{t}$ and $qm_0+a \in J$), and note that we may replace $s(m)$ in~\eqref{eq:q-bad-prog} by $s(m_0)$ with a negligible error, since $\varphi$ has Lipschitz norm at most $1$ and
\[ d(s(m)\widetilde{g}(m)\gamma(m), s(m_0)\widetilde{g}(m)\gamma(m)) = d(s(m), s(m_0)) \ll M^{-1} \]
for all $m$ with $qm+a \in J$ by the right invariance of $d$ and the smoothness property of $s$. Moreover, by the periodicity of $\gamma$, we may replace $\gamma(m)$ in~\eqref{eq:q-bad-prog} by $\gamma(m_0)$. Now let $g'$ be the polynomial sequence defined by
\[ g'(m) = \gamma(m_0)^{-1} \widetilde{g}(tm + b) \gamma(m_0), \]
taking values in $G' = \gamma(m_0)^{-1} \widetilde{G} \gamma(m_0)$, and let $\varphi'$ be the automorphic function on $G'$ defined by
\[ \varphi'(x) = \varphi(s(m_0) \gamma(m_0) x). \]
The desired properties about $G'/\Gamma', g', \varphi'$ can be established via standard ``quantitative nil-linear algebra" (see the claim at the end of~\cite[Section 2]{GT-mobius-nil}). After a change of variables replacing $m$ by $tm + b$, the inequality~\eqref{eq:q-bad-prog} can be rewritten as
\[ \left| \sum_{m \colon qtm + c \in J} f(qtm + c) \varphi'(g'(m)) \right| \gg \ee \frac{|J|}{qt}, \]
where $c = qb + a$. This is almost what we need, but there is the slight issue that $c$ may not be coprime with $t$. Let $d = (c, t)$ so that $d \leq M$. Let $q' = qt/d$ and $a' = c/d$ so that $(a',q') = 1$. Let $I = d^{-1}J$ so that $|I| \gg X/M^3$. Since $f$ is completely multiplicative, we have
\[ \left| \sum_{m \colon q'm + a' \in  I} f(q'm + a') \varphi'(g'(m)) \right| \gg \ee \frac{|I|}{q'}. \]
This completes the proof of the lemma.
\end{proof}

\section{The minor arc case: Proof of Proposition~\ref{prop:BV-equidist}}\label{sec:BV-equidist}

In this section we prove Proposition~\ref{prop:BV-equidist}, which is the minor arc case of our main theorem and applies to all $1$-bounded multiplicative functions. For convenience write 
\[ T = \sum_{q \in \CQ} \left(\frac{|I_q|}{q} + 1\right). \]
We may assume that $(X/Q^2)^{-1/20} < \eta < (\log X)^{-C}$ for some sufficiently large $C = C(d,s) > 0$, since otherwise the bound is trivial. We may further assume that $|I_q| \geq X^{0.9}$ for each $q \in \CQ$, since the contributions from those $q$ with $|I_q| \leq X^{0.9}$ are trivially acceptable. After multiplying each $\varphi_q$ by an appropriate scalar, it suffices to prove the desired inequality with the absolute value sign removed. Set
\[ Y = \eta^{-1}, \ \ Z = (X/Q^2)^{1/20}, \]
so that $2 \leq Y < Z \leq X^{1/20}$. Let $F$ be the function defined by
\[ F(n) =  \sum_{\substack{q \in \CQ \\ n \in I_q \\ n \equiv a_q\pmod{q}}} \psi_q((n-a_q)/q). \]
Clearly $F$ is supported on $[0,X]$. The desired bound can be rewritten as
\begin{equation}\label{eq:f-F-orth} 
\sum_{n \leq X} f(n) F(n) \ll \frac{\log Y}{\log Z} \cdot T + \eta^c X\log X \max_{q \in \CQ} \|\varphi_q\|_{\text{Lip}(\CX_q)}. 
\end{equation}
As alluded to in the introduction, this will be proved using an orthogonality criterion for multiplicative functions. A general principle of this type is given in~\cite[Proposition 2.2]{Gre16}. In the notations there, the terms giving rise to $E_{\text{triv}}$ and $E_{\text{sieve}}$ will be dealt with by Lemma~\ref{lem:ramare-l1bound} and Lemma~\ref{lem:E-sieve}, respectively. In particular, the $E_{\text{sieve}}$ term leads to the first bound in~\eqref{eq:f-F-orth}. The bilinear (type-II) sum $E_{\text{bilinear}}$ will be dealt with in Lemma~\ref{lem:typeII}, leading to the second bound in~\eqref{eq:f-F-orth}.

Unfortunately we cannot directly apply~\cite[Proposition 2.2]{Gre16}, since for example our function $F$ is not necessarily bounded. In the remainder of this section we reproduce the argument from~\cite{Gre16} with suitable modifications to prove~\eqref{eq:f-F-orth}. Recall the definition of Ramar{\'e}'s weight function:
\[ w(n) = \frac{1}{\#\{Y \leq p < Z: p \mid n\} + 1}. \]
Introduce also the function $\mu_{[Y,Z)}^2$, which is the indicator function of the set of integers $n$ that is not divisible by the square of any prime $p \in [Y,Z)$. To prove~\eqref{eq:f-F-orth}, we first dispose of those terms with $\mu_{[Y,Z)}^2(n) = 0$:
\[ \sum_{\substack{n \leq X \\ \mu_{[Y,Z)}^2(n) = 0}} |f(n) F(n)| \leq \sum_{Y \leq p < Z} \sum_{\substack{n \leq X \\ p^2 \mid n}} |F(n)|. \]
The following lemma will be used repeatedly.

\begin{lemma}\label{lem:ramare-l1bound}
For any positive integer $D \leq X^{0.4}$ we have
\[ \sum_{\substack{n \leq X \\ D \mid n}} |F(n)| \ll \frac{T}{D}. \]
\end{lemma} 

\begin{proof}
Using the trivial bound
\begin{equation}\label{eq:F-trivial} 
|F(n)| \leq \sum_{\substack{q \in \CQ \\ n \in I_q \\ n \equiv a_q\pmod{q}}} 1,
\end{equation}
we obtain
\[ \sum_{\substack{n \leq X \\ D \mid n}} |F(n)|  \leq \sum_{q \in \CQ} \sum_{\substack{n \in I_q \\ D \mid n \\ n \equiv a_q\pmod{q}}} 1. \]
Since $(a_q,q) = 1$, the inner sum over $n$ is nonempty unless $(q,D) = 1$, in which case it is $O(|I_q|/qD)$. The conclusion follows immediately.
\end{proof}

Since $Z^2 \leq X^{0.4}$, Lemma~\ref{lem:ramare-l1bound} implies that
\[ \sum_{\substack{n \leq X \\ \mu_{[Y,Z)}^2(n) = 0}} |f(n)F(n)| \ll T\sum_{Y \leq p < Z} \frac{1}{p^2} \ll \frac{T}{Y}. \]
Hence the contributions from those $n \leq X$ with $\mu_{[Y,Z)}^2(n) = 0$ are acceptable. If $\mu_{[Y,Z)}^2(n) = 1$, then we have the Ramar{\'e} identity
\[ \sum_{\substack{Y \leq p < Z \\ p \mid n}} w(n/p) = \begin{cases} 1 & \text{if }p \mid n\text{ for some }Y \leq p < Z \\ 0 & \text{otherwise.} \end{cases} \]
The following lemma disposes of those $n$ not divisible by any $p \in [Y, Z)$:

\begin{lemma}\label{lem:E-sieve}
We have
\[ \sum_{n \leq X} |F(n)| \cdot \1_{(n,\prod_{Y \leq p < Z}p) = 1} \ll \frac{\log Y}{\log Z} \cdot T. \]
\end{lemma}

\begin{proof}
Using~\eqref{eq:F-trivial}, we can bound the left hand side by
\[ \sum_{q \in \CQ} \sum_{\substack{n \in I_q \\ n \equiv a_q \pmod{q}}} \1_{(n,\prod_{Y \leq p < Z}p) = 1}. \]
Consider the inner sum for a fixed $q$. Writing $d = (q, \prod_{Y \leq p < Z} p)$, we may bound the inner sum using a standard upper bound sieve (since $Z^2 \leq |I_q|/q$) to obtain
\[ \sum_{\substack{n \in I_q \\ n \equiv a_q \pmod{q}}} \1_{(n,\prod_{Y \leq p < Z}p) = 1} \ll \frac{|I_q|}{q} \prod_{\substack{Y \leq p < Z \\ p\nmid q}} \left(1 - \frac{1}{p}\right) \ll \frac{|I_q|}{q} \cdot \frac{\log Y}{\log Z} \cdot \frac{d}{\varphi(d)}.  \]
On the other hand, since $d \mid \prod_{Y \leq p < Z} p$ and $d \leq q$ we have
\[ \frac{d}{\varphi(d)} \leq \prod_{Y \leq p < W} \left(1 - \frac{1}{p}\right)^{-1}, \]
where $W \sim Y + \log q$. Thus $d/\varphi(d) = O(1)$ since $Y \geq \log q$, and the conclusion of the lemma follows.
\end{proof}

Thus we can restrict to those $n$ with $\mu_{[Y,Z)}^2(n) = 1$ and having at least one prime divisor $p \in [Y,Z)$. By the Ramar{\'e} identity, we need to estimate
\[ \Sigma := \sum_{\substack{n \leq X \\ \mu_{[Y,Z)}^2(n) = 1}} f(n) F(n) \sum_{\substack{Y \leq p < Z \\ p\mid n}} w(n/p). \]
Writing $m = n/p$ and using the multiplicativity of $f$, we obtain
\begin{equation}\label{eq:Sigma} 
\Sigma = \sum_{\substack{m \leq X/Y \\ \mu_{[Y,Z)}^2(m) = 1}} w(m) f(m) \sum_{\substack{Y \leq p < Z \\ p \leq X/m \\ (m,p)=1}} f(p) F(pm). 
\end{equation}
The condition $\mu_{[Y,Z)}^2(m) = 1$ can be dropped since the contribution from those $m$ divisible by $\widetilde{p}^2$ for some $\widetilde{p} \in [Y,Z)$ is at most
\[ \sum_{Y \leq \widetilde{p} < Z} \sum_{Y \leq p < Z} \sum_{\substack{m \leq X/p \\ \widetilde{p}^2 \mid m}} |F(pm)| \ll T \sum_{Y \leq \widetilde{p} < Z} \sum_{Y \leq p < Z} \frac{1}{p\widetilde{p}^2}  \ll \frac{T}{Y} \log\log X, \]
by an application of Lemma~\ref{lem:ramare-l1bound}. Similarly, the condition $(m,p) = 1$ in~\eqref{eq:Sigma} can also be dropped since the contribution from the terms with $p \mid m$ is at most
\[ \sum_{Y \leq p < Z} \sum_{\substack{m \leq X/p \\ p\mid m}} |F(pm)| \ll T \sum_{Y \leq p < Z} \frac{1}{p^2} \ll \frac{T}{Y}, \]
by Lemma~\ref{lem:ramare-l1bound}. Both these bounds are acceptable. Thus it remains to bound
\[ \Sigma' := \sum_{m \leq X/Y} w(m) f(m) \sum_{\substack{Y \leq p < Z \\ p \leq X/m}} f(p) F(pm). \]
Dyadically dividing the range $[Y,Z)$ for $p$, we consider
\[ \Sigma'(P) := \sum_{m \leq X/P} w(m) f(m) \sum_{\substack{P \leq p < 2P \\ p \leq X/m}} f(p) F(pm) \]
for $P \in [Y, Z)$. Use the trivial bound $|w(m) f(m)| \leq 1$ and apply the Cauchy-Schwarz inequality to obtain
\[ |\Sigma'(P)|^2 \ll \frac{X}{P} \sum_{m \leq X/P} \left|\sum_{\substack{P \leq p < 2P \\ p \leq X/m}} f(p) F(pm) \right|^2. \]
%Since $X^{1/2} \geq Z^8$, the $L^2$-norm of $w$ above can be bounded by $(X/P) (\log\log X)^{-2}$. 
After expanding the square and changing the order of summation, we obtain
\begin{align*} 
|\Sigma'(P)|^2 & \ll \frac{X}{P} \sum_{P \leq p,p' < 2P} f(p) \overline{f(p')} \sum_{m \leq \min(X/p, X/p')} F(pm) \overline{F(p'm)} \\
& \ll \frac{X}{P} \sum_{P \leq p,p' < 2P} \left| \sum_{m \leq \min(X/p, X/p')} F(pm) \overline{F(p'm)} \right|.
\end{align*}
Set $K = P$, $L = X/P$, and $\delta = \eta^c$ for some $c>0$ small enough depending on $d,s$. The following lemma, whose proof will be given in Section~\ref{sec:type-II}, gives the necessary estimates for the type-II (bilinear) sums appearing above.

\begin{lemma}\label{lem:typeII}
Let $K, L, Q \geq 2$ be parameters with $10Q^2 \leq L$. Let $\delta \in (0,1/2)$. Associated to each $Q \leq q < 2Q$ we have:
\begin{enumerate}
\item a residue class $a_q\pmod q$ with $0 \leq a_q < q, (a_q,q) = 1$, and an arbitrary interval $I_q$;
\item a nilmanifold $G_q/\Gamma_q$ of dimension at most some $d \geq 1$, equipped with a filtration $(G_q)_{\bullet}$ of degree at most some $s \geq 1$ and a $\delta^{-1}$-rational Mal'cev basis $\CX_q$;
\item a polynomial sequence $g_q: \Z \to G_q$ adapted to $(G_q)_{\bullet}$;
\item a Lipschitz function $\varphi_q: G_q/\Gamma_q \to \C$ with $\|\varphi_q\|_{\text{Lip}(\CX_q)} \leq 1$ and $\int\varphi_q = 0$.
\end{enumerate}
Let $\psi_q: \Z \to \C$ be the function defined by $\psi_q(n) = \varphi_q(g_q(n)\Gamma_q)$, and let $F$ be the function defined by
\[ F(n) = \sum_{\substack{Q \leq  q < 2Q \\ n \in I_q \\ n\equiv a_q \pmod{q}}} \psi_q((n-a_q)/q). \]
For each $k,k' \in [K, 2K)$, let $I(k,k') \subset [0,L]$ be an arbitrary interval. Suppose that
\begin{equation}\label{eq:large-typeII} 
\sum_{K \leq k, k' < 2K} \left| \sum_{\ell \in I(k,k')} F(k\ell) \overline{F(k'\ell)} \right| \geq \delta K^2 L,
\end{equation}
and that $K^{-c} < \delta < (\log Q)^{-1}$ for some sufficiently small $c = c(d,s) > 0$. Then the polynomial sequence $\{g_q(m)\}_{1 \leq m \leq KL/Q}$ fails to be totally $\delta^{O_{d,s}(1)}$-equidistributed for some $Q \leq q < 2Q$.
\end{lemma}

To complete the proof of Proposition~\ref{prop:BV-equidist}, note that the hypotheses $10Q^2 \leq L$ and $K^{-c} < \delta < (\log Q)^{-1}$ in Lemma~\ref{lem:typeII} are satisfied by our choices of $Y$ and $Z$. By setting $\varphi_q = 0$ for $q \notin \CQ$ and renormalizing (replacing $\varphi_q$ by $\varphi_q/\|\varphi_q\|_{\text{Lip}(\CX_q)}$), we may apply Lemma~\ref{lem:typeII} to conclude that
\[ |\Sigma'(P)|^2 \ll \frac{X}{P} \cdot \eta^c PX \max_{q \in \CQ} \|\varphi_q\|_{\text{Lip}(\CX_q)}^2. \]
The desired bound~\eqref{eq:f-F-orth} follows after summing over $P$ dyadically.

\begin{remark}
Instead of using simply the trivial bound $|w(n)| \leq 1$, one may appeal to~\cite[Lemma 2.1]{Gre16} to dispose of the extra $\log X$ factor that appeared when summing over $P$ dyadically. We will, however, not bother with this since the type-II estimates we use already have an extra logarithmic factor anyways.
\end{remark}

\section{Type-II estimates}\label{sec:type-II}

In this section we prove Lemma~\ref{lem:typeII}. We start with the following lemma, needed to treat composite moduli.

\begin{lemma}\label{lem:typical-q}
Let $Q \geq 2$ and $Q \leq R \leq 4Q^2$. Let $E$ be the set of pairs $(q,q')$ with $Q \leq q,q' < 2Q$ and $R \leq [q,q'] < 2R$. For each $Q \leq q < 2Q$ and $R \leq r < 2R$, let
\[ m_q(r) = \#\{Q \leq q' < 2Q: (q,q') \in E, [q,q'] = r\}. \]
Then for any $m_0 \geq 1$ we have
\[ \#\{(q,q') \in E: m_q([q,q']) \geq m_0\} \ll m_0^{-1}R\log Q. \]
\end{lemma}

\begin{proof}
By a dyadic division, it suffices to show that
\[ \#\{(q,q') \in E: m_0 \leq m_q([q,q']) < 2m_0\} \ll m_0^{-1}R\log Q \]
for any $m_0 \geq 1$. Call the left hand side above $N(m_0)$. For any $D \geq 1$, let $\sigma_D(q)$ be the number of divisors of $q$ in the range $[D, 8D]$. A moment's thought reveals that $m_q(r) \leq \sigma_D(q)$ where $D = Q^2/2R$. Indeed, each $q'$ with $[q,q'] = r$ gives rise to a divisor $(q,q')$ of $q$ in the range $[D, 8D]$, and moreover $(q,q')$ is uniquely determined by $q'$ via $(q,q') = qq'/r$. It follows that
\[ N(m_0) = \sum_{\substack{Q \leq q < 2Q \\ \sigma_D(q) \geq m_0}} \sum_{\substack{R \leq r < 2R \\ m_0 \leq m_q(r) < 2m_0}} m_q(r).  \]
Since $m_q(r) = 0$ unless $q \mid r$, the inner sum over $r$ is $O(m_0R/Q)$. It thus suffices to show that
\[ \#\{Q \leq q < 2Q: \sigma_D(q) \geq m_0\} \ll m_0^{-2}Q\log Q \]
for any $D \geq 1$. We may assume that $D \leq Q^{1/2}$ since otherwise we may replace $D$ by $Q/8D$. By the second moment method, we have
\[ \#\{Q \leq q < 2Q: \sigma_D(q) \geq m_0\} \leq \frac{1}{m_0^2} \sum_{Q \leq q < 2Q} \sigma_D(q)^2. \]
After expanding out the square and changing the order of summation, the right hand side above is
\[ \frac{1}{m_0^2} \sum_{D \leq d_1, d_2 \leq 8D} \sum_{\substack{Q \leq q < 2Q \\ [d_1,d_2] \mid q}} 1 \ll \frac{Q}{m_0^2} \sum_{D \leq d_1, d_2 \leq 8D} \frac{1}{[d_1,d_2]} \ll \frac{Q}{m_0^2} \log D. \]
This completes the proof of the lemma.
\end{proof}

%\begin{lemma}\label{lem:typeII}
%Let $K, L, Q \geq 2$ be parameters with $10Q^2 \leq L$. Let $\delta \in (0,1/2)$. Associated to each $Q \leq q < 2Q$ we have:
%\begin{enumerate}
%\item a residue class $a_q\pmod q$ with $0 \leq a_q < q, (a_q,q) = 1$, and an arbitrary interval $I_q$;
%\item a nilmanifold $G_q/\Gamma_q$ of dimension at most some $d \geq 1$, equipped with a filtration of degree at most some $s \geq 1$ and a $\delta^{-1}$-rational Mal'cev basis $\CX_q$;
%\item a polynomial sequence $g_q: \Z \to G_q$;
%\item a Lipschitz function $\varphi_q: G_q/\Gamma_q \to \C$ with $\|\varphi_q\|_{\text{Lip}(\CX_q)} \leq 1$ and $\int\varphi_q = 0$.
%\end{enumerate}
%Write $\psi_q = \varphi_q \circ g_q$, and let $F$ be the function defined by
%\[ F(n) = \sum_{\substack{Q \leq  q < 2Q \\ n \in I_q \\ n\equiv a_q \pmod{q}}} \psi_q((n-a_q)/q). \]
%For each $k,k' \in [K, 2K)$, let $I(k,k') \subset [0,L]$ be an arbitrary interval. Suppose that
%\begin{equation}\label{eq:large-typeII} 
%\sum_{K \leq k, k' < 2K} \left| \sum_{\ell \in I(k,k')} F(k\ell) \overline{F(k'\ell)} \right| \geq \delta K^2 L,
%\end{equation}
%and that $K^{-c} < \delta < (\log Q)^{-1}$ for some sufficiently small $c = c(d,s) > 0$. Then the polynomial sequence $\{g_q(m)\}_{1 \leq m \leq KL/Q}$ fails to be totally $\delta^{O_{d,s}(1)}$-equidistributed for some $Q \leq q < 2Q$.
%\end{lemma}

There are two places in the proof of Lemma~\ref{lem:typeII} where we lose a factor of $\log Q$ (and hence the assumption that $\delta < (\log Q)^{-1}$). One place is from dyadically decomposing the possible values of $[q,q']$, and the other from the conclusion of Lemma~\ref{lem:typical-q}. If one is only interested in prime moduli, then this extra loss can certainly be saved.

\begin{proof}[Proof of Lemma~\ref{lem:typeII}]
In this proof, all implied constants are allowed to depend on $d,s$. For $k,k' \in [K, 2K)$, we may write
\begin{equation}\label{eq:typeII-proof1} 
\sum_{\ell \in I(k,k')} F(k\ell) \overline{F(k'\ell)}  = \sum_{Q \leq q,q' < 2Q} \sum_{\substack{\ell \in I(k,k',q,q') \\ k\ell \equiv a_q\pmod{q} \\ k'\ell \equiv a_{q'} \pmod{q'}}} \psi_q((k\ell-a_q)/q) \overline{\psi_{q'}((k'\ell - a_{q'})/q')}, 
\end{equation}
for some interval $I(k,k',q,q') \subset I(k,k')$. The solution to the simultaneous congruence conditions
\[ k\ell \equiv a_q\pmod{q}, \ \ k'\ell \equiv a_{q'} \pmod{q'} \]
takes the form
\[ \ell \equiv a(k,k',q,q') \pmod{[q,q']}, \]
for some $0 \leq a(k,k',q,q') < [q,q']$. It is possible that no solutions exist, in which case we may simply set $I(k,k',q,q')$ to be empty and assign an arbitrary value to $a(k,k',q,q')$. After a change of variables $\ell = [q,q']m + a(k,k',q,q')$, the inner sum over $\ell$ in~\eqref{eq:typeII-proof1} can be rewritten as
\[ \sum_{m \in J(k,k',q,q')} \psi_q\left(\frac{k[q,q']}{q}m + b \right) \overline{\psi_{q'}\left(\frac{k'[q,q']}{q'}m + b' \right)}, \]
for some interval $J(k,k',q,q') \subset [0, L/[q,q']]$, where
\[ b = \frac{1}{q}(k a(k,k',q,q') - a_q), \ \ b' = \frac{1}{q'}(k' a(k,k',q,q') - a_{q'}). \]
In principle $b,b'$ depend on $k,k',q,q'$, but to simplify notations we drop this dependence, as the precise nature of $b,b'$ is unimportant, apart from the obvious facts that $0 \leq b \leq k[q,q']/q$ and $0 \leq b' \leq k'[q,q']/q'$. Consider the polynomial sequence $g_{k,k',q,q'}: \Z \to G_q\times G_{q'}$ defined by
\[ g_{k,k',q,q'}(m) = \left(g_q\left(\frac{k[q,q']}{q}m + b\right), g_{q'}\left(\frac{k'[q,q']}{q'}m + b' \right) \right), \]
and the Lipschitz function $\varphi_{q,q'}: G_q/\Gamma_q \times G_{q'}/\Gamma_{q'} \to \C$ defined by
\[ \varphi_{q,q'}(x, x') = \varphi_q(x) \overline{\varphi_{q'}(x')}. \]
Then the type-II sum from~\eqref{eq:typeII-proof1} can be written as
\[ \sum_{\ell \in I(k,k')} F(k\ell) \overline{F(k'\ell)}  = \sum_{Q \leq q,q' < 2Q} \sum_{m \in J(k,k',q,q')} \varphi_{q,q'}(g_{k,k',q,q'}(m)). \]
After dyadically dividing the possible values of $[q,q']$, we deduce from the hypothesis~\eqref{eq:large-typeII} that
\begin{equation}\label{eq:large-sum-R} 
\sum_{K \leq k,k' < 2K} \sum_{\substack{Q \leq q,q' < 2Q \\ R \leq [q,q'] < 2R}} \left| \sum_{m \in J(k,k',q,q')} \varphi_{q,q'}(g_{k,k',q,q'}(m)) \right| \gg \delta^2 K^2 L,
\end{equation}
for some $Q \leq R \leq 4Q^2$, where we used the assumption that $\delta < (\log Q)^{-1}$. For the rest of the proof fix such a $R$. Hence there is a subset $T$ consisting of quadruples $(k,k',q,q')$ with $R \leq [q,q'] < 2R$, such that
\begin{equation}\label{eq:typeII-T} 
|T| \gg \delta^{O(1)} K^2R, 
\end{equation}
and for $(k,k',q,q') \in T$ we have
\begin{equation}\label{eq:pp'qq'} 
\left| \sum_{m \in J(k,k',q,q')} \varphi_{q,q'}(g_{k,k',q,q'}(m)) \right| \gg \frac{\delta^2 L}{R}. 
\end{equation}
Since $\int \varphi_{q,q'} = 0$, the inequality~\eqref{eq:pp'qq'} implies that the sequence $\{g_{k,k',q,q'}(m)\}_{0 \leq m \leq L/R}$ fails to be $\delta^{O(1)}$-equidistributed. Hence by \cite[Theorem 2.9]{GT-nilsequence}, there is a nontrivial horizontal character $\chi_{q,q'} = \chi_{k,k',q,q'}: G_q\times G_{q'} \to \C$ with $\|\chi_{q,q'}\| \ll \delta^{-O(1)}$, such that
\begin{equation}\label{eq:chiqq'} 
\|\chi_{q,q'} \circ g_{k,k',q,q'}\|_{C^{\infty}(L/R)} \ll \delta^{-O(1)}.
\end{equation}
We have tacitly assumed that $\chi_{q,q'}$ is independent of $k,k'$, since this can be achieved after pigeonholing in the $\delta^{-O(1)}$ possible choices of $\chi_{q,q'}$ and enlarging the constant $O(1)$ in~\eqref{eq:typeII-T} appropriately. More explicitly, if we write
\begin{equation}\label{eq:betai-kq} 
\chi_{q,q'} \circ g_{k,k',q,q'}(m) = \sum_{i=0}^s \beta_i(k,k',q,q') m^i
\end{equation}
for some coefficients $\beta_i(k,k',q,q') \in \R$, then~\eqref{eq:chiqq'} combined with~\cite[Lemma 3.2]{GT-mobius-nil} implies that there is a positive integer $r = O(1)$ such that
\begin{equation}\label{eq:betai} 
\|r \beta_i(k,k',q,q')\| \ll \delta^{-O(1)} (L/R)^{-i} 
\end{equation}
for each $1 \leq i \leq s$ and $(k,k',q,q') \in T$.  Write $\chi_{q,q'} = (\chi_{q,q'}^{(1)}, \chi_{q,q'}^{(2)})$, where $\chi_{q,q'}^{(1)},\chi_{q,q'}^{(2)}$ are horizontal characters on $G_q,G_{q'}$, respectively, with $\|\chi_{q,q'}^{(1)}\| \ll \delta^{-O(1)}$ and $\|\chi_{q,q'}^{(2)}\| \ll \delta^{-O(1)}$. Write also
\[
\chi_{q,q'}^{(1)} \circ g_q(n) = \sum_{i=0}^s \alpha_i(q,q') n^i, \ \ \chi_{q,q'}^{(2)} \circ g_{q'}(n) = \sum_{i=0}^s \alpha_i'(q,q') n^i, \]
for some coefficients $\alpha_i(q,q'), \alpha_i'(q,q') \in \R$. 

\begin{claim}
There exists a sequence of subsets $E_1 \subset \cdots \subset E_s$ of pairs $(q,q')$ with $R \leq [q,q'] < 2R$ and a sequence of positive integers $r_1 \geq \cdots \geq r_s$ with $r_{i+1} \mid r_i$ for each $i$, such that $|E_1| \gg \delta^{O(1)}R$, $r_1 \ll \delta^{-O(1)}$, and moreover
\[ \|r_i \alpha_i(q,q')\| \ll \delta^{-O(1)} (KL/Q)^{-i}, \ \ \|r_i \alpha_i'(q,q')\| \ll \delta^{-O(1)} (KL/Q)^{-i}, \]
for each $1 \leq i \leq s$ and $(q,q') \in E_i$ .
\end{claim}

Note that the claim actually implies the bounds
\[ \|r_i \alpha_j(q,q')\| \ll \delta^{-O(1)} (KL/Q)^{-j}, \ \ \|r_i \alpha_j'(q,q')\| \ll \delta^{-O(1)} (KL/Q)^{-j}, \]
for each $1 \leq i \leq j \leq s$ and $(q,q') \in E_j$.

Assuming the claim, we may conclude the proof of the lemma as follows. Pick an arbitrary pair $(q,q') \in E_1$. Since $\chi_{q,q'}$ is nontrivial, either $\chi_{q,q'}^{(1)}$ or $\chi_{q,q'}^{(2)}$ is nontrivial. Without loss of generality, assume that $\chi_{q,q'}^{(1)}$ is nontrivial. The diophantine information about $\alpha_i(q,q')$ from the claim implies that
\[ \|r_1 \chi_{q,q'}^{(1)} \circ g_q\|_{C^{\infty}(KL/Q)} \ll \delta^{-O(1)}. \]
Thus by~\cite[Lemma 5.3]{FH}, the polynomial sequence $g_q$ fails to be totally $\delta^{O(1)}$-equidistributed. 

It remains to establish the claim. Start by finding a subset $E$ of pairs $(q,q')$ with $R \leq [q,q'] < 2R$, such that the following properties hold:
\begin{enumerate}
\item $|E| \gg \delta^{O(1)}R$;
\item  for each pair $(q,q') \in E$, there are at least $\delta^{O(1)}K^2$ pairs $(k,k')$ with $(k,k',q,q') \in T$;
\item  for each pair $(q,q') \in E$, there are at most $\delta^{-O(1)}$ pairs $(q,\widetilde{q}') \in E$ with $[q,q'] = [q,\widetilde{q}']$, and similarly there are at most $\delta^{-O(1)}$ pairs $(\widetilde{q}, q') \in E$ with $[q,q'] = [\widetilde{q}, q']$.
\end{enumerate}
Indeed, from the bound~\eqref{eq:typeII-T} we may first find $E$ satisfying (1) and (2), and then apply Lemma~\ref{lem:typical-q} with $m_0 = \delta^{-C}$ for some sufficiently large $C$ to remove a small number of pairs from $E$, so that property (3) is satisfied.

Construct $\{E_i\}$ and $\{r_i\}$ in the claim by downward induction on $i$ as follows. Take $E_{s+1} = E$ just constructed and $r_{s+1} = r$ from~\eqref{eq:betai}. Now let $1 \leq i \leq s$, and suppose that $E_j, r_j$ have already been constructed for $j > i$ satisfying the desired properties. First we show that for each pair $(q,q') \in E_{i+1}$, there is a positive integer $\widetilde{r}(q,q') \ll \delta^{-O(1)}$ such that
\begin{equation}\label{eq:alpha-diophantine-1} 
\left\| \widetilde{r}(q,q') r_{i+1} \alpha_i(q,q') \left(\frac{[q,q']}{q}\right)^i \right\| \ll \delta^{-O(1)} (KL/R)^{-i}, 
\end{equation}
and similarly  with $\alpha_i(q,q')$ replaced by $\alpha_i'(q,q')$. To prove this, fix a pair $(q,q')\in E_{i+1}$, and for the purpose of simplifying notations we drop the dependence on $q,q'$ so that
\[ \alpha_i = \alpha_i(q,q'), \ \ \alpha_i' = \alpha_i'(q,q'), \ \ \beta_i(k,k') = \beta_i(k,k',q,q'). \]
From the definition of $g_{k,k',q,q'}$ we see the following relationship between the coefficients $\alpha_i, \alpha_i', \beta_i(k,k')$:
\begin{equation}\label{eq:beta-alpha-relation}
\beta_i(k,k') = \left(\frac{k[q,q']}{q}\right)^i \sum_{i \leq j \leq s} \alpha_j  b^{j-i} + \left(\frac{k'[q,q']}{q'}\right)^i  \sum_{i \leq j \leq s} \alpha_j' b'^{j-i}. 
\end{equation}
Write $\widetilde{\beta}_i(k,k')$ for the contribution from the term with $j=i$:
\[ \widetilde{\beta}_i(k,k') =  \left(\frac{k[q,q']}{q}\right)^i \alpha_i + \left(\frac{k'[q,q']}{q'}\right)^i \alpha_i'. \]
By the induction hypothesis, $\|r_{i+1}\alpha_j\|$ and $\|r_{i+1}\alpha_j'\|$ are small for $j > i$. Combined with the bound $0 \leq b,b' \ll KR/Q$, this implies that the terms with $j>i$ are negligible:
\[ \left\|r_{i+1} (\beta_i(k,k') - \widetilde{\beta}_i(k,k')) \right\| \ll \delta^{-O(1)}(L/R)^{-i}. \] 
It follows from~\eqref{eq:betai} that
\begin{equation}\label{eq:coeffi}  
\left\|r_{i+1} \widetilde{\beta}_i(k,k') \right\| \ll \delta^{-O(1)}(L/R)^{-i},
\end{equation}
whenever $(k,k',q,q') \in T$. Since $(q,q') \in E$, this holds for at least $\delta^{O(1)} K^2$ pairs $(k,k')$. Choose $k'$ such that~\eqref{eq:coeffi} holds whenever $k \in \CK$, for some subset $\CK$ with $|\CK| \gg \delta^{O(1)}K$. Since
\[ \widetilde{\beta}_i(k,k') - \widetilde{\beta}_i(\widetilde{k},k') = \alpha_i \left(\frac{[q,q']}{q}\right)^i (k^i - \widetilde{k}^i), \]
it follows that for $k,\widetilde{k} \in \CK$ we have
\[ \left\| r_{i+1} \alpha_i \left(\frac{[q,q']}{q}\right)^i (k^i - \widetilde{k}^i) \right\| \ll \delta^{-O(1)} (L/R)^{-i}.  \]
Since $\delta > K^{-c}$ for some sufficiently small $c>0$, the desired inequality~\eqref{eq:alpha-diophantine-1} follows from a standard recurrence result such as~\cite[Lemma 4.5]{GT-nilsequence}. The analogous bound for $\alpha_i'$ can be proved in a similar way. 

Now that we have established~\eqref{eq:alpha-diophantine-1}, define $\widetilde{E}_i \subset E_{i+1}$ to be a subset with $|\widetilde{E}_i| \gg \delta^{O(1)} |E_{i+1}| \gg \delta^{O(1)}R$, such that $\widetilde{r}(q,q')$ for $(q,q') \in \widetilde{E}_i$ take a common value $\widetilde{r}$. We say that a pair $(q,q') \in \widetilde{E}_i$ is \emph{typical}, if there are at least $\delta^{O(1)}R/Q$ pairs $(q,\widetilde{q}') \in \widetilde{E}_i$ with $\chi_{q,q'}^{(1)} = \chi_{q,\widetilde{q}'}^{(1)}$, and similarly there are at least $\delta^{O(1)}R/Q$ pairs $(\widetilde{q},q') \in \widetilde{E}_i$ with $\chi_{q,q'}^{(2)} = \chi_{\widetilde{q},q'}^{(2)}$.  Define $E_i \subset \widetilde{E}_i$ to be the set of typical pairs in $\widetilde{E}_i$. By choosing the constant $O(1)$ in the definition of typical pairs sufficiently large, we may ensure that $|E_i| \gg \delta^{O(1)}R$.

Now let $(q,q') \in E_i$. Since $(q,q')$ is typical, there exists a subset $\CQ(q,q')$ with $|\CQ(q,q')| \gg \delta^{O(1)}R/Q$, such that $(q,\widetilde{q}') \in \widetilde{E}_i$ and $\chi_{q,q'}^{(1)} = \chi_{q,\widetilde{q}'}^{(1)}$ for all $\widetilde{q}' \in \CQ(q,q')$. Thus $\alpha(q,q') = \alpha(q,\widetilde{q}')$ for all $\widetilde{q}' \in \CQ(q,q')$, and by~\eqref{eq:alpha-diophantine-1} applied to $(q,\widetilde{q}')$ we obtain
\[ \left\| \widetilde{r} r_{i+1} \alpha_i(q,q') \left(\frac{[q,\widetilde{q}']}{q}\right)^i \right\| \ll \delta^{-O(1)} (KL/R)^{-i}, \]
for each $\widetilde{q}' \in \CQ(q,q')$. Since $\widetilde{E}_i \subset E$, property (3) of the set $E$ implies that
\[ \# \{[q,\widetilde{q}']/q  \colon \widetilde{q}' \in \CQ(q,q') \} \gg \delta^{O(1)}|\CQ(q,q')| \gg \delta^{O(1)} R/Q. \]
By a standard recurrence result such as \cite[Lemma 4.5]{GT-nilsequence}, there exists $r_i \ll \delta^{-O(1)}\widetilde{r}r_{i+1} \ll \delta^{-O(1)}$ such that
\[ \|r_i \alpha_i(q,q')\| \ll \delta^{-O(1)} (KL/Q)^{-i}, \]
as desired. The analogous bound for $\alpha_i'(q,q')$ can be proved in a similar way. This finishes the proof of the claim, and also the proof of Lemma~\ref{lem:typeII}.
\end{proof}

\bibliographystyle{plain}
\bibliography{mobius-progression}

\end{document}